\documentclass[12pt,oneside,reqno]{amsart}

\usepackage{hyperref}

\usepackage[headsep=30pt,footskip=30pt,twoside=false,bottom=2.5cm]{geometry}

\usepackage[all]{xy}

\xymatrixcolsep{2cm}

\numberwithin{equation}{section}

\allowdisplaybreaks[1]

\usepackage{etoolbox}

\makeatletter

\patchcmd{\thesubsection}{\arabic}{\arabic}{}{}

\patchcmd{\@seccntformat}{\@secnumfont}{%
  \@secnumfont\expandafter\protect\csname format#1\endcsname}{}{}

\patchcmd{\@startsection}{\@afterindenttrue}{\@afterindentfalse}{}{}

\patchcmd{\subsection}{-.5em}{.3\linespacing}{}{}

\makeatother

\theoremstyle{plain}

\newtheorem{theorem}{Theorem}[section]

\newtheorem{proposition}[theorem]{Proposition}

\newtheorem{lemma}[theorem]{Lemma}
\newtheorem{question}[theorem]{Question}

\newtheorem{corollary}[theorem]{Corollary}

\theoremstyle{remark}

\newtheorem{remark}[theorem]{Remark}

\newcommand{\Ker}[1]{\ensuremath{\mathrm{Ker} (#1)}}

\newcommand{\SKer}[1]{\ensuremath{\mathcal{K}er (#1)}}

\newcommand{\Aut}[1]{\ensuremath{\mathrm{Aut} (#1)}}

\newcommand{\Pic}[1]{\ensuremath{\mathrm{Pic} (#1)}}

\newcommand{\At}[1]{\ensuremath{\mathrm{At} (#1)}}

\newcommand{\ENd}[1]{\ensuremath{\mathrm{End}  (#1)}}

\newcommand{\ParENd}[1]{\ensuremath{\mathrm{ParEnd}  (#1)}}

\newcommand{\SParENd}[1]{\ensuremath{\mathrm{SParEnd}  (#1)}}

\newcommand{\Parend}[1]{\ensuremath{\mathrm{ParEnd'}  (#1)}}

\newcommand{\SParend}[1]{\ensuremath{\mathrm{SParEnd'}  (#1)}}
\newcommand{\Img}[1]{\ensuremath{\mathrm{Im} (#1)}}

\newcommand{\cat}[1]{\ensuremath{\mathcal{#1}}}

\newcommand{\codim}[1]{\ensuremath{\mathrm{codim}(#1)}}

\newcommand{\End}[2][]{\ensuremath{\mathrm{End}_{#1} (#2)}}

\newcommand{\END}[2][]{\ensuremath{\mathcal{E}\mathit{nd}_{#1} (#2)}}

\newcommand{\id}[1]{\ensuremath{\mathbf{1}_{#1}}}

\renewcommand{\dim}[2][]{\ensuremath{\mathrm{dim}_{#1}(#2)}}

\newcommand{\rk}[2][]{\ensuremath{\mathrm{rk}_{#1}(#2)}}

\newcommand{\Z}{\ensuremath{\mathbb{Z}}}

\newcommand{\Q}{\ensuremath{\mathbb{Q}}}

\newcommand{\p}{\ensuremath{\mathbf{P}}}

\newcommand{\set}[1]{\ensuremath{\{ #1 \}}}

\newcommand{\C}{\ensuremath{\mathbb{C}}}

\newcommand{\tr}[1]{\ensuremath{\mathrm{Tr}(#1)}}

\newcommand{\struct}[1]{\ensuremath{\mathcal{O}_{#1}}}

\newcommand{\DIFF}[4][]{\ensuremath{\mathcal{D}\mathit{iff}^{#1}_{#2}(#3,#4)}}  
  
\newcommand{\Diff}[4][]{\ensuremath{\mathrm{Diff}^{#1}_{#2}(#3,#4)}}

\newcommand{\coh}[3]{\ensuremath{\mathrm{H}^{#1}(#2,#3)}}

\begin{document}

\title[Moduli space of Parabolic connections]{Line bundles on the moduli space of Parabolic connections
over a compact Riemann surface}

\author{Anoop Singh}

\address{School of Mathematics,
Tata Institute of
Fundamental Research \\
Mumbai, India, 400005}
  \email{anoops@math.tifr.res.in}


\maketitle

\begin{abstract}
Let $X$ be a compact Riemann surface of genus $g \geq 3$ 
and $S$  a finite subset of $X$.   Let $\xi$ be fixed a holomorphic line bundle over $X$ of degree $d$. 
Let $\cat{M}_{pc}(r, d, \alpha)$ (respectively, $\cat{M}_{pc}(r, \alpha, \xi)$ ) denote the moduli space of parabolic connections 
 of rank $r$, degree $d$ and full flag rational
generic weight system $\alpha$, (respectively, with the fixed 
  determinant $\xi$ )
 singular over the parabolic points $S \subset X$. 
 Let 
 $\cat{M}'_{pc}(r, d, \alpha)$ (respectively, $\cat{M}'_{pc}(r, \alpha, \xi)$)   be the Zariski dense open subset  of $\cat{M}_{pc}(r, d, \alpha)$ (respectively, $\cat{M}_{pc}(r, \alpha, \xi)$ )parametrizing all parabolic connections such that the underlying parabolic bundle is stable. 
 We show that there is a natural compactification of the 
 moduli spaces $\cat{M}'_{pc}(r, d, \alpha)$, and 
 $\cat{M}'_{pc}(r, \alpha, \xi)$ by smooth divisors.
 We describe the numerically effectiveness of these divisors at infinity.
 We determine the Picard group of the moduli spaces $\cat{M}_{pc}(r, d, \alpha)$, and $\cat{M}_{pc}(r, \alpha, \xi)$.
 Let $\cat{C}(L)$ denote the space  of holomorphic connections
  on  an ample line bundle $L$ over the moduli space $\cat{M}(r, d, \alpha)$ of parabolic bundles. We  show that $\cat{C}(L)$ does not admit any non-constant algebraic function. 
 \end{abstract}

\section{Introduction and statements of the results}
The moduli space of parabolic $\Lambda$-modules over a smooth projective curve has been constructed in \cite{A17}, which is a generalization of the moduli space of 
$\Lambda$-modules constructed in \cite{S1}, \cite{S2} by Simpson. The moduli space of parabolic connections over a smooth projective curve 
is the moduli space of parabolic $\Lambda$-modules where 
$\Lambda$ is the sheaf of rings of differential operators. The moduli space of parabolic connections is also constructed in \cite{I}.

Let $X$ be a compact Riemann surface of genus $g \geq 3$, and $S = \{x_1, \ldots, x_m\}$ the finite subset of 
$X$, which we call the set of parabolic points.
Let $E$ be a holomorphic vector bundle over $X$ of rank 
$r$ and degree $d$. See section \ref{pre} for the definition of  parabolic weights $\alpha$ (parabolic  structure) on $E$, parabolic connections
on a parabolic vector bundle $E_*$, parabolic Higgs bundles and their moduli spaces. 
We consider the full flag rational generic parabolic weights.
Let $$\alpha = \{\alpha^x_1, \ldots, \alpha^x_r \}_{x \in S}$$ be the fixed generic system of rational parabolic weights
corresponding to the full flag filtration such that 
$$  \sum_{x \in S} \sum_{i = 1}^r \alpha_i^x \in \Z .$$

Let $\xi$ be a fixed holomorphic line bundle over $X$ of degree 
$d$ such that 
\begin{equation}
\label{eq:c13}
d = \deg{\xi} = - \sum_{x \in S} \sum_{i = 1}^r \alpha_i^x.
\end{equation}
Let $\cat{M}(r, d, \alpha)$ (respectively, $\cat{M}(r, \alpha, \xi)$) denote  the moduli space of stable parabolic vector bundles of rank $r$, degree $d$ and 
full flag rational generic system of weights $\alpha$ (respectively, with determinant $\xi$).

Let $\xi$ be given a trivial filtration at each $x \in S$ with parabolic weight 
\begin{equation*}
\label{eq:b14}
\beta^x := \sum_{i=1}^r \alpha_i^x.
\end{equation*}
We denote this parabolic line bundle by $\xi_*$.
Now, since $p\deg(\xi_*) = 0$, from \cite[Theorem 3.1]{B1},
$\xi_*$ admits a parabolic connection $D_{\xi_*}$
such that $Res(D_{\xi_*}, x) = \beta^x$ for every 
$x \in S$.

 Let $\cat{M}_{pc}(r, d, \alpha)$  (respectively, $\cat{M}_{pc}(r, \alpha, \xi)$) be the moduli space of 
 stable parabolic connections  $(E_*, D)$ of rank $r$, degree $d$
 and parabolic weight $\alpha$ (respectively, 
 with  $({\bigwedge^r E}_*, \tilde{D}) \cong (\xi_*, D_{\xi_*})$, where $\tilde{D}$ is the parabolic connection on the parabolic vector bundle ${\bigwedge^r E}_*$ induced by $D$).
 
 Let $\cat{M}'_{pc}(r, d, \alpha)$ (resp. $\cat{M}'_{pc}(r, \alpha, \xi)$) be the open dense subset of 
 $\cat{M}_{pc}(r, d, \alpha)$ (resp. $\cat{M}_{pc}(r, \alpha, \xi)$) consists of those parabolic connections whose underlying vector bundle is stable.

In \cite{LS},  
the Picard group of the moduli space of parabolic vector 
bundles has been computed.
Now using the techniques from \cite{BR}, and 
\cite{AS1}, we show the following  (see Theorem \ref{thm:1}) $$\Pic{\cat{M}_{pc}(r, d, \alpha)} \cong \Pic{\cat{M}(r, d, \alpha)}.$$

While proving the isomorphism between Picard groups, we show that there
is a natural compactification of the moduli space 
$\cat{M}'_{pc}(r, d, \alpha)$ by a smooth divisor.

Let $L$ be an ample line bundle over $\cat{M}(r, d, \alpha)$, and $\cat{C}(L)$ denote the space of all holomorphic connections on $L$. Then,
$\cat{C}(L)$ is a $T^*\cat{M}(r, d, \alpha)$-torsor,
where $T^*\cat{M}(r, d, \alpha)$ is the cotangent bundle 
of $\cat{M}(r, d, \alpha)$. 
Then, we have the following (see Proposition \ref{prop:2})
$$\Pic{\cat{C}(L)} \cong \Pic{\cat{M}(r, d, \alpha)}.$$

Also, we show that   the global regular functions on the variety $\cat{C}(L)$ are constant functions  (see Theorem \ref{thm:2}), that is,
$$\coh{0}{\cat{C}(L)}{\struct{\cat{C}(L)}} = \C.$$

Let $\At{L}$ denote the Atiyah bundle of $L$ (see \eqref{eq:a40}). In order to prove the Theorem \ref{thm:2}, 
we have shown the following 
$$\coh{0}{\cat{M}}{\cat{S}ym^k\At{L}} = \C,$$
for every $k \geq 0$, where $\cat{S}ym^k\At{L}$ denote the 
symmetric powers of the Atiyah bundle $\At{L}$.

Let $\p (\cat{V})$ be the compactification of the moduli space $\cat{M}'_{pc}(r, \alpha, \xi)$ (see Proposition \ref{prop:3}) with the complement $${\bf H}_0
  = \p (\cat{V}) \setminus \cat{M}'_{pc}(r, \alpha, \xi)$$ a smooth divisor at infinity.
  Then, we prove that ${\bf H}_0$ is numerically effective if and only if the tangent bundle 
  $T \cat{M}(r, \alpha, \xi)$ is numerically effective
  (see Proposition \ref{prop:4}).

\section{Preliminaries}
\label{pre}
 We recall the notion of parabolic vector bundles and parabolic connections on a parabolic vector bundle over a compact Riemann surface.
 Let $X$ be a compact Riemann surface of genus $g \geq 3$. Let $S = \{x_1, \ldots, x_m\} $ be a finite subset of $X$, which we call the set of parabolic points. Let 
 $E$ be a holomorphic vector bundle over $X$.
 
 A \emph{quasi-parabolic structure} on $E$ at a point $x\in X$ is a
 strictly decreasing flag 
\begin{equation*}
	\label{eq:a1}
 E_x = E^1_x \supseteq E^2_x \supseteq \dotsb \supseteq E^k_x 
  \supseteq E^{k+1}_x = 0
\end{equation*}
of linear subspaces in the fibre $E_x$.  We set
\begin{equation*}
	\label{eq:a2}
 m^x_j = \dim[\C]{E^j_x} - \dim[\C]{E^{j+1}_x}.
\end{equation*}
The integer $k$ is called the \emph{length} of the flag and the
$k$-tuple $(m^x_1, \dotsc, m^x_k)$ is called the \emph{type}
of the flag.  We say that the flag is a full flag if $m^x_j = 1$ for
all $1 \leq j \leq k$.  A \emph{parabolic structure} in $E$ at $x$ is
just a quasi-parabolic structure at $x$ together with a sequence of
real numbers $$0 \leq \alpha^x_1 < \dotsb < \alpha^x_k < 1.$$  The real
numbers $\alpha^x_j$ are called the \emph{weights}.
We denote by $\alpha = \{(\alpha_1^x, \ldots, \alpha_k^x)\}_{x \in S}$ the system of real weights 
corresponding to a fixed parabolic structure.

A \emph{parabolic vector bundle} with parabolic structure on $S$ is a
 holomorphic vector bundle $E$ together with a parabolic structure in
 $E$ at each point $x \in S$.  We shall write $E_*$ to denote the
 parabolic vector bundle with underlying vector bundle $E$.

   For a
 parabolic bundle $E_*$ with system of weights $\alpha$ the \emph{parabolic degree} is defined to be
 the real number
 \begin{equation*}
	\label{eq:a3}
  p\deg(E_*) = \deg(E) + \sum_{x \in D} 
   \sum_{j = 1}^k m^x_j \alpha^x_j,
 \end{equation*}
 where $\deg(E)$ denotes the degree of $E$, and we put
 \begin{equation*}
	\label{eq:a4}
  p\mu(E_*) = \frac{p\deg(E_*)}{\rk{E}},
 \end{equation*}
 where $\rk{E}$ is the rank of $E$.  The real number $p\mu(E_*)$ is
 called the parabolic slope of $E_*$.

 Let $E_*$ be a parabolic bundle and let $F$ be a vector subbundle of 
 $E$.  Then the parabolic structure on $E$ induces a parabolic
 structure on $F$ as follows: first we take the induced filtration on
 $F$.  Next, for each $x \in S$ and for every
 $j \in \set{1, \dotsc, k_x}$, we set
 $\alpha^x_j(F) := \alpha^x_i(E)$, where $i$ is the largest 
 integer such that $F^j_x \subset E^i_x$ and $k_x$ is the length of 
 the flag in $F_x$.  The vector bundle $F$ together with this
 parabolic structure is denoted by $F_*$ and is called a parabolic
 subbundle of $E$.

 A parabolic bundle $E_*$ is said to be \emph{parabolic semistable} if
 for every non-zero proper parabolic subbundle $F_*$ we have  
 \begin{equation}
	\label{eq:a5}
 p\mu(F_*) \leq p\mu(E_*).
 \end{equation}
 The parabolic bundle $E_*$ is said to be \emph{parabolic stable} if
 all the inequalities in \eqref{eq:a5} are strict.

In what follows, we consider the full flag filtration.
We fix the integers $r \geq 1$ and $d \in \Z$.
We can represent  the system  of  weights $\alpha = \{(\alpha_1^x, \ldots, \alpha_k^x)\}_{x \in S}$ in matrix form $(\alpha^{x_i}_{j})^{1 \leq i \leq m}_{1 \leq j \leq r} \in \Q^{mr}$. 
For the definition of generic system of weights $\alpha$; see \cite[Definition 2.2]{AG17}, and for more details
see  \cite{BY}. 
Let
$$W^{(m)}_{r}(d) := \{ \alpha = (\alpha^{x_i}_{j})^{1 \leq i \leq m}_{1 \leq j \leq r} \in \Q^{mr} \, \vert \, \alpha \,\text{is generic and} \, d + \sum_{i =1}^m \sum_{j =1}^r \alpha^{x_i}_j = 0  \}. $$
In \cite[Definition 2.2]{I}, Inaba defined the system of special weights, and the system of weights which is not special is called generic.
See \cite[section 2.2]{BH}
for the non-emptiness of the space $W^{(m)}_{r}(d)$ of admissible parabolic weights. 

Let $$\alpha = (\alpha^x_1, \ldots, \alpha^x_r )_{x \in S} = (\alpha^{x_i}_{j})^{1 \leq i \leq m}_{1 \leq j \leq r} \in W^{(m)}_{r}(d)$$ be the fixed system of generic  rational parabolic weights
corresponding to the full flag filtration.
Let $\cat{M}^{ss}(r,d,\alpha)$ be the moduli space of 
semi-stable parabolic vector bundles of rank $r$, degree $d$ and weight system $\alpha$ (see \cite{MS}). The  moduli space 
$\mathcal{M}^{ss}(r,d, \alpha)$ is a normal projective variety
of dimension 
$$\dim{\cat{M}^{ss}(r,d,\alpha)} = r^2(g-1) +1 + \frac{m(r^2-r)}{2}.$$ 
Moreover,  the moduli space $\cat{M}^s(r,d, \alpha)$ of stable parabolic
bundles is an smooth and open subset of $\cat{M}^{ss}(r,d,\alpha)$. 

Let $\xi$ be a fixed holomorphic line bundle over $X$
of degree $d$ such that 
\begin{equation}
\label{eq:b13}
d = \deg{\xi} = - \sum_{x \in S} \sum_{i = 1}^r \alpha_i^x.
\end{equation}
Let $\cat{M}^{ss}(r,\alpha, \xi)$ be the moduli space of semi-stable parabolic vector bundles on $X$ of rank $r$
and  determinant $\xi$, that is, 
$\bigwedge^r E \cong \xi$, with weight system $\alpha$.
Then $\cat{M}^{ss}(r, \alpha, \xi)$ is a projective variety of dimension (see \cite[p.n. 557]{BH})
$$ (r^2-1)(g-1) + \frac{m(r^2-r)}{2}.$$

Let $\cat{M}^s(r,\alpha,\xi) \subset \cat{M}^{ss}(r,\alpha,\xi)$ be the open subset parametrizing the stable parabolic bundles. This is an irreducible smooth  variety.

Since the full flag weight system $\alpha$
is generic (see \cite[Definition 2.2]{AG17}), we have 
$$\cat{M}(r, d, \alpha) := \cat{M}^{ss}(r, d, \alpha) = \cat{M}^{s}(r, d, \alpha),$$
and  
$$\cat{M}(r,\alpha, \xi) := \cat{M}^{ss}(r, \alpha, \xi) = \cat{M}^{s}(r, \alpha, \xi).$$

Let $E_*$ be a parabolic bundle. We say that 
an endomorphism 
$\phi : E \longrightarrow E$ is {\it strongly parabolic}
 if for every $x \in S$, we have 
 $$\phi (E^i_x) \subset E^{i+1}_x.$$
 Similarly, $\phi : E \longrightarrow E$ is said to be 
 {\it weakly 
 parabolic} or just {\it parabolic} if it satisfies 
 $$\phi (E^i_x) \subset E^i_x.$$
 
The sheaf of strongly parabolic endomorphism on $E_*$ is 
denoted by $\SParENd{E_*}$ and the sheaf of parabolic endomorphism on $E_*$ is denoted by $\ParENd{E_*}$.

Now, we recall the notion of logarithmic connections (see \cite{BM} and \cite{D})
and its residues, and using these notions we define 
parabolic connections.
Let $E$ be a holomorphic vector bundle over $X$.
Let $$S = x_1 + \cdots + x_m$$ be the reduced effective
divisor associated with $S$.
A {\bf logarithmic connection}  on $E$ singular over $S$ is a $
\C$-linear map 
\begin{equation}
\label{eq:a6}
D : E \to E \otimes \Omega^1_X(S)  = E \otimes
\Omega^1_X \otimes \struct{X}(S)  
\end{equation}
which satisfies the Leibniz identity
\begin{equation}
\label{eq:a7}
D(f s)= f D(s) + df \otimes s,
\end{equation}
where $f$ is a local section of \struct{X} and $s$ is a 
local section of $E$.

We next describe the notion of residues of a
logarithmic connection $D$ in $E$ singular over 
$S$.  

Let $v \in E_{x_\beta}$ be any vector in the fiber of $E$
over $x_\beta \in S$. Let $U$ be an open set around $x_\beta$ and
$s: U \to E$ be a holomorphic section of $E$ over $U$ 
such that $s(x_\beta) = v$. Consider the following 
composition 
\begin{equation*}
\label{eq:a9}
\Gamma(U,E) \to \Gamma(U, E \otimes\Omega^1_X \otimes 
\struct{X}(S)) \to (E \otimes\Omega^1_X \otimes \struct{X}
(S))_{x_\beta} = E_{x_\beta},
\end{equation*}
where the equality is given because 
for any $x_\beta \in S$, the fibre $(\Omega^1_X 
\otimes \struct{X}(S))_{x_\beta}$ is canonically 
identified with $\C$ by sending a meromorphic form to 
its residue at $x_\beta$.
Then, we have an endomorphism on $E_{x_\beta}$
sending $v$ to $D(s)(x_\beta)$. We need to check that 
this endomorphism is well defined. Let $s' : U  \to E$ be another holomorphic section  such that $s'(x_\beta) = v$. Then 
 $$(s - s')(x_\beta) = v - v = 0.$$
 Let $t$ be a local coordinate at $x_\beta$ on $U$ such that 
$t(x_\beta) = 0$, that is, the coordinate system $(U,t)$ is centered at $x_\beta$. 
Since  $ s - s' \in 
\Gamma(U,E)$ and  $(s - s')(x_\beta) = 0$, $s - s' = t  \sigma$ 
for some $\sigma \in \Gamma(U,E)$. Now,
\begin{align*}
D( s - s') = D(t \sigma)& = t D(\sigma) + dt \otimes \sigma \\
& = t D(\sigma) + t (\frac{dt}{t} \otimes \sigma),
\end{align*}
and hence $D( s - s')(x_\beta) = 0$, that is,
$D(s)(x_\beta) = D(s')(x_\beta)$.

 Thus, we have a well defined 
endomorphism, denoted by 
\begin{equation}
\label{eq:a10}
Res(D,x_\beta) \in \ENd{E}_{x_\beta} = \ENd{E_{x_\beta}}
\end{equation}
that sends $v$ to $D(s)(x_\beta)$.  This endomorphism 
$Res(D,x_\beta)$ is called the \textbf{residue} of the logarithmic 
connection $D$ at the point $x_\beta \in S$ (see \cite{D} for the details). 

From \cite[Theorem 3]{O}, for a logarithmic connection
$D$ singular over $S$, we have
\begin{equation}
\label{eq:a11}
\deg{E}  + \sum_{j=1}^m \tr{Res(D,x_j)} = 0,
\end{equation}
where, $\deg{E}$ denotes the degree of $E$, and
$\tr{Res(D,x_j)}$ denote the trace of the  
endomorphism $Res(D,x_j) \in \ENd{E_{x_j}}$, for all 
$j =1, \ldots, m$.

Let $E_*$ be a parabolic vector bundle over $X$ with 
   a fixed system of weights $\alpha$ for full flag filtration. A parabolic 
connection (for the group $\text{GL}(r,\C)$) on $E_*$  is a logarithmic 
connection $D$ on the underlying vector bundle $E$
satisfying following conditions:
\begin{enumerate}
\item \label{a} For each $x \in S$ the homomorphism 
induced in the filtration over the fibre
$E_x$ satisfies 
\begin{equation}
\label{eq:a12}
D(E^i_x) \subset E^i_x \otimes \Omega^1_X(S)\vert_x
\end{equation}
for every $i = 1, \ldots, r$.

\item \label{b} For every $x \in S$ and for every 
$i = 1, \ldots, r$ the action of $Res(D,x)$ on the quotient ${E^i_x}/{E^{i+1}_x}$ is the multiplication by 
$\alpha_i^x$. Since $Res(D,x)$ preserves the filtration 
it acts on each quotient.
\end{enumerate}

We denote the parabolic connections by the pair 
$(E_*, D)$.
 
Let $\xi$ be the fixed line bundle  of degree $d$ and weight $\alpha$
such that 
\begin{equation}
\label{eq:a13}
d = \deg{\xi} = - \sum_{x \in S} \sum_{i = 1}^r \alpha_i^x.
\end{equation}
We want to fix a parabolic connection on $\xi$. We first equip $\xi$ with a parabolic structure as follows.
Let $\xi$ be given a trivial filtration at each $x \in S$ with parabolic weight 
\begin{equation*}
\label{eq:a14}
\beta^x := \sum_{i=1}^r \alpha_i^x.
\end{equation*}
We denote $\xi$ with a parabolic structure by $\xi_*$.
Then, from \eqref{eq:a13}, the parabolic degree of $\xi_*$ is
\begin{equation*}
\label{eq:a15}
p\deg(\xi_*) = \deg(\xi) + \sum_{x\in S} \beta^x
= \deg(\xi) + \sum_{x \in S} \sum_{i = 1}^r \alpha_i^x = 0.
\end{equation*}
Since $p\deg(\xi_*) = 0$, from \cite[Theorem 3.1]{B1},
$\xi_*$ admits a parabolic connection $D_{\xi_*}$
such that $Res(D_{\xi_*}, x) = \beta(x)$ for every 
$x \in S$. Now, we fix a pair $(\xi_*, D_{\xi_*})$ as 
described above.

Next, given a pair $(E_*, D)$, we say that $D$ is a parabolic connection (for the group $\text{SL}(r, \C)$ )
if the logarithmic connection $$\tr{D}: \bigwedge^r E 
\longrightarrow \bigwedge^r E \otimes \Omega^1_X(S) $$
induced from $D$
coincides with $D_{\xi_*}$, that is, 
we have an isomorphism 
\begin{equation}
\label{eq:a15.1}
((\bigwedge^r E)_*, \tr{D}) \cong (\xi_*, D_{\xi_*}).
\end{equation}
A parabolic connection $(E_*, D)$ (for the group
$\text{GL}(r, \C)$ or $\text{SL}(r, \C)$) is said to be semi-stable (respectively, stable) if for every non-zero 
proper parabolic subbundle $F_*$ of $E_*$, which is 
invariant under $D$, that is, 
$D(F) \subset F \otimes \Omega^1_X (S)$, we have 
\begin{equation*}
\label{eq:a16}
p\mu(F_*) \leq p\mu(E_*) ~~ \text{(respectively,} <). 
\end{equation*} 

Note that $(E_*, D)$
is stable does not imply that $E_*$ is stable.

Recall that a parabolic connection $(E_*, D)$ is said to be {\bf irreducible} if there does not exists a non-zero proper parabolic subbundle of $E_*$ which is invariant 
under $D$.

We have an useful result as follows.
\begin{lemma}
\label{lem:a}
Suppose that a parabolic connection $(E_*, D)$ is 
irreducible.  Then $(E_*, D)$ is stable.
\end{lemma}

Let $(E_*, D)$ and $(E'_*, D')$ be two full flag parabolic connections with same generic weight $\alpha$.
A morphism between parabolic connections $(E_*, D)$ and $(E'_*, D)$ is a parabolic  morphism (already defined)
$$\phi :E_* \longrightarrow E'_*$$ of parabolic vector bundles such that 
the following diagram involving logarithmic connections
\begin{equation}
\label{eq:mor}
\xymatrix@C=4em{
E \ar[r]^{D} \ar[d]^{\phi} & E \otimes \Omega^1_X(S) \ar[d]^{\phi \otimes \id{\Omega^1_X(S)}} \\
E' \ar[r]^{D'} & E'\otimes \Omega^1_X(S)}
\end{equation}
commutes. Moreover, we say that $(E_*, D)$
and $(E'_*, D')$ are isomorphic if $\phi$
is an isomorphism.

\begin{lemma}
\label{lem:1.a} Let 
$(E_*, D)$ and $(E'_*, D')$ be semi-stable full flag parabolic connections over $X$ with same  generic weight $\alpha$.
Then we have the followings. 
\begin{enumerate}
\item \label{a1} Suppose $(E_*, D)$ and $(E'_*, D')$ are stable and $p\mu(E_*) = p\mu(E'_*).$
If $$\phi : (E_*, D) \to (E'_*, D')$$  is a non-zero morphism of parabolic connections, then it is an isomorphism.
\item \label{a2} If $(E_*, D)$ is stable, then the only parabolic endomorphisms  of the pair  $(E_*, D)$ are scalars.
\end{enumerate}
\end{lemma}
\begin{proof}\mbox{}
\begin{enumerate}
\item First note that a subbundle and its quotient bundle of a parabolic vector bundle admits a parabolic structure from its original parabolic bundle (see the remark after Definition 1.8 of \cite{MS}). We shall use usual technique of Kernel-Image and Image-Coimage sequences to prove the first part.

Let $\Ker{\phi}_* \subset E_*$ is a parabolic subbundle.
Then, $\Ker{\phi}_*$ is $D$-invariant parabolic subbundle of $E$. Since $(E_*, D)$ is stable, we have $p\mu(\Ker{\phi}_*) < p\mu(E_*)$. Since $\phi \neq 0$, we have $\Img{\phi} \neq 0$. Now, $\Img{\phi}$ inherits a parabolic structure from $E'_*$.

Consider the Kernel-Image short exact sequence 
$$0 \to \Ker{\phi}_* \to E_* \to \Img{\phi}_* \to 0.$$
Then, $p\mu(E_*) < p\mu(\Img{\phi}_*)$.
Next, consider the Image-Coimage short exact sequence 
$$ 0 \to \Img{\phi}_* \to E'_* \to (E' / \Img{\phi})_* \to 0.$$
Note that $\Img{\phi}_*$ is a $D'$-invariant parabolic subbundle of $E'$. Since $(E'_*, D')$ is stable, we have 
$p\mu(\Img{\phi}_*) < p\mu (E'_*)$. Thus, we get that 
$p\mu(E_*) < p\mu(E'_*)$ which contradicts the assumption that
$p\mu(E_*) = p\mu (E'_*)$. Therefore, on underlying vector bundles, we have $\Ker{\phi} = 0$ and 
$\Img{\phi} = E'$.

\item Let  $\psi : (E_*, D) \to (E_*, D)$ be the given endomorphism. Let $x \in X$, and let $\lambda \in \C$ be an eigen value of the linear map $\psi(x) : E_x \to E_x$. Note that $X$ is a compact Riemann surface, it does not admit any non-constant holomorphic function, therefore the eigen values and their multiplicities are independent of $x \in X$.
Since $D$ is $\C$-linear, 
$\psi - \lambda \id{E_*}$ is an endomorphism of $(E_*, D)$, that is, 
$$D \circ (\psi - \lambda \id{E_*}) = (\psi - \lambda \id{E_*}) \otimes \id{\Omega^1_X(S)} \circ D.$$
Since $(E_*, D)$ is stable, 
from \eqref{a1}, $\psi - \lambda \id{E_*}$ is either a zero morphism or an isomorphism. 
Since $\lambda$ is an eigen value of $\psi(x)$, kernel of 
$\psi - \lambda \id{E_*}$ is non-trivial. Thus, $\psi = \lambda \id{E_*}$, where $\lambda \in \C$.

\end{enumerate}
\end{proof}

Let $\cat{M}^{ss}_{pc}(r,d, \alpha)$ denote the moduli space of semi-stable parabolic connections for the group \text{GL}$(r,\C)$ of rank $r$, degree 
$d$ and weight system $\alpha$. Let 
$$\cat{M}^{sm}_{pc}(r, d, \alpha) \subset \cat{M}^{ss}_{pc}(r,d, \alpha)$$ be the smooth locus of $\cat{M}^{ss}_{pc}(r,d, \alpha)$. Let $\cat{M}_{pc}(r,d, \alpha) $ be the 
moduli space of stable parabolic connections. Then, from
 \cite[Theorem 2.1]{I}, 
$\cat{M}_{pc}(r,d, \alpha)$
is a smooth irreducible quasi-projective variety
of dimension $2r^2(g-1) +2 + m(r^2-r)$, and hence 
$\cat{M}_{pc}(r,d, \alpha) \subset \cat{M}^{sm}_{pc}(r,d, \alpha) $.

Let 
\begin{equation}
\label{eq:a20.5}
\iota : \cat{M}'_{pc}(r,d, \alpha) \hookrightarrow \cat{M}_{pc}(r, d, \alpha) 
\end{equation}
be the natural inclusion, where $\cat{M}'_{pc}(r,d, \alpha)$
is the open subvariety consists of the pairs $(E_*, D)$
whose underlying vector bundle $E_*$ is stable.

Let
\begin{equation}
\label{eq:a20.6}
\pi : \cat{M}'_{pc}(r, d, \alpha) \longrightarrow \cat{M}(r, d, \alpha)
\end{equation}
be a map defined by sending $(E_*, D)$ to $E_*$.
In other words, $\pi$ is the forgetful map which forgets the
parabolic connection. 
Note that $\pi$ is a surjective morphism follows from 
\cite[Theorem 3.1]{B1}.
Now, for every $E_* \in \cat{M}(r, d, \alpha)$,
$\pi^{-1}(E_*)$ is an affine space modelled over 
$\coh{0}
{X}{\Omega^1_X(S) \otimes \SParENd{E_*}}$ which is described as follows.
Given two parabolic connections $D$ and $D'$
on a parabolic vector bundle $E_*$, 
the difference $D - D'$ is an $\struct{X}$-module homomorphism from $E$ to $E \otimes \Omega^1_X(S)$
such that the residues $$Res(D,x) - Res(D',x) = Res(D-D', x)$$ acts as the zero morphism on each successive 
quotients $E^i_x / E^{i+1}_x$. Therefore, for every 
$x \in S$, we have 
$$(D-D')(E^i_x) \subset E^{i+1}_x \otimes \Omega^1_X(S)\vert_{x}.$$
Conversely, given any parabolic connection 
$(E_*, D)$ and $$\Phi \in \coh{0}
{X}{\Omega^1_X(S) \otimes \SParENd{E_*}},$$
$D + \Phi$ is again a parabolic connection on $E_*$.
Thus, the space $\pi^{-1}(E_*)$ of parabolic connections
on $E_*$ forms an 
affine space modelled over the vector space 
$\coh{0}
{X}{\Omega^1_X(S) \otimes \SParENd{E_*}}.$

Now,
by deformation theory, the tangent space
$T_{E_*}\cat{M}(r, d, \alpha)$ at $E_*$ 
is isomorphic to 
$\coh{1}{X}{\ParENd{E_*}}.$
By parabolic version of Serre duality, we have 
\begin{equation}
\label{eq:a21.2}
\coh{1}{X}{\ParENd{E_*}}^{\vee} \cong \coh{0}
{X}{\Omega^1_X(S) \otimes \SParENd{E_*}},
\end{equation}
therefore the cotangent space $T^*_{E_*}\cat{M}(r, d, \alpha)$ at $E_*$ is isomorphic to $\coh{0}
{X}{\Omega^1_X(S) \otimes \SParENd{E_*}}$.

Thus, $T^*_{E_*}\cat{M}(r, d, \alpha)$ acts on the fibre $\pi^{-1}(E_*)$ faithfully and transitively, which proves the following Lemma.
 
\begin{lemma}
\label{lem:1} The moduli space
$\cat{M}'_{pc}(r, d, \alpha)$ is a $T^*\cat{M}(r, d, \alpha)$-torsor.
\end{lemma}

Let $\cat{M}_{pc}(r,d,\xi)$ be the moduli space of stable parabolic connections (for the group \text{SL}$(r, \C)$) of rank $r$, degree 
$d$, generic weight system $\alpha$, and with fixed 
determinant $(\xi_*, D_{\xi_*})$.

The moduli space
$\cat{M}_{pc}(r,d,\xi)$ is a smooth irreducible quasi-projective variety of dimension $2(g-1)(r^2-1) + m (r^2-r)$ (see \cite[Proposition 5.1]{I}).

Let 
\begin{equation}
\label{eq:a21}
\cat{M}'_{pc}(r, \alpha, \xi) \subset \cat{M}_{pc}(r, \alpha, \xi)
\end{equation}
be the subset containing $(E_*, D)$ whose underlying parabolic vector bundle $E_*$ is stable.
Then, from [\cite{M} Theorem 2.8(A)], $\cat{M}'_{pc}(r, \alpha, \xi)$ is a  Zariski open subset of $\cat{M}_{pc}(r, \alpha, \xi)$.

Let $\SParend{E_*}$ and $\Parend{E_*}$ denote respectively the sheaves of trace zero strongly and weakly parabolic endomorphisms on $E_*$.

Let
\begin{equation}
\label{eq:a21.1}
\pi_{\xi} : \cat{M}'_{pc}(r, \alpha, \xi) \longrightarrow \cat{M}(r,\alpha, \xi)
\end{equation}
be the map defined by sending $(E_*, D)$ to $E_*$, that is, $\pi_{\xi}$ is the forgetful map.
Again from \cite[Theorem 3.1]{B1},  $\pi_{\xi}$ is a surjective morphism .
Now, for every $E_* \in \cat{M}(r,\alpha, \xi)$,
$\pi^{-1}_{\xi}(E_*)$ is an affine space modelled over 
$\coh{0}
{X}{\Omega^1_X(S) \otimes \SParend{E_*}}$ which is described as follows.
Given two parabolic connections $D$ and $D'$
on a parabolic vector bundle $E_*$, 
the difference $D - D'$ is an $\struct{X}$-module homomorphism from $E$ to $E \otimes \Omega^1_X(S)$
such that the residues $$Res(D,x) - Res(D',x) = Res(D-D', x)$$ acts as the zero morphism on each successive 
quotients $E^i_x / E^{i+1}_x$. Therefore, for every 
$x \in S$, we have 
$$(D-D')(E^i_x) \subset E^{i+1}_x \otimes \Omega^1_X(S)\vert_{x}.$$
Moreover, since the determinant of parabolic connections  are fixed, we have $$\tr{D-D'} = 0.$$
Conversely, given any parabolic connection 
$(E_*, D)$ and $$\Phi \in \coh{0}
{X}{\Omega^1_X(S) \otimes \SParend{E_*}},$$
that is, $\Phi$ is strongly parabolic morphism 
with $\tr{\Phi} = 0$. Then,
$D + \Phi$ is again a parabolic connection on $E_*$.
Thus, the space $\pi^{-1}(E_*)$ of parabolic connections
on $E_*$ with fixed determinant  forms an 
affine space modelled over the vector space 
$\coh{0}
{X}{\Omega^1_X(S) \otimes \SParend{E_*}}.$

Now,
by deformation theory, the tangent space
$T_{E_*}\cat{M}(r, \alpha, \xi)$ at $E_*$ 
is isomorphic to 
$\coh{1}{X}{\Parend{E_*}}.$
By parabolic version of Serre duality, we have 
\begin{equation}
\label{eq:a21.3}
\coh{1}{X}{\Parend{E_*}}^{\vee} \cong \coh{0}
{X}{\Omega^1_X(S) \otimes \SParend{E_*}},
\end{equation}
therefore the cotangent space $T^*_{E_*}\cat{M}(r, \alpha, \xi)$ at $E_*$ is isomorphic to $\coh{0}
{X}{\Omega^1_X(S) \otimes \SParend{E_*}}$.

Thus, $T^*_{E_*}\cat{M}(r, \alpha, \xi)$ acts on the fibre $\pi^{-1}_{\xi}(E_*)$ faithfully and transitively, which proves the following Lemma. 
\begin{lemma}
\label{lem:2} The moduli space
$\cat{M}'_{pc}(r, \alpha, \xi)$ is a $T^*\cat{M}(r, \alpha, \xi)$-torsor.
\end{lemma}
 
A {\bf strongly} parabolic Higgs bundle $(E_*, \Phi)$ is a 
parabolic vector bundle $E_*$ together with an $\struct{X}$-module homomorphism 
$$\Phi : E \longrightarrow E \otimes \Omega^1_X \otimes \struct{X}(S)$$
such that 
for every $x \in S$ the homomorphism induced in the 
filtration over the fibre $E_x$ satisfies
$$\Phi (E^i_x) \subset E^{i+1}_x \otimes \Omega^1_X(S)\vert_{x},$$
where  
$\Omega^1_X (S)$ is the line bundle $\Omega^1_X \otimes \struct{X}(S)$.
$\Phi$ is called parabolic Higgs field on $E_*$.
Similarly, we can define {\bf weakly} parabolic Higgs bundle to be a pair $(E_*, \Phi)$ where $E_*$ is a parabolic vector bundle and 
$$\Phi : E \longrightarrow E \otimes \Omega^1_X \otimes \struct{X}(S)$$
is an \struct{X}-module homomorphism
such that for very $x \in S$, we have
$$\Phi (E^i_x) \subset E^{i}_x \otimes \Omega^1_X(S)\vert_{x}.$$

A parabolic Higgs bundle $(E_*, \Phi)$ (either strongly or weakly),  is said to be semi-stable (respectively, stable) if for every non-zero 
proper parabolic subbundle $F_*$ of $E_*$, which is 
invariant under $\Phi$, that is, 
$\Phi(F) \subset F \otimes \Omega^1_X (S)$, we have 
\begin{equation*}
\label{eq:a11.1}
p\mu(F_*) \leq p\mu(E_*) ~~ \text{(respectively,} <). 
\end{equation*} 
 If we do not say explicitly if a parabolic Higgs 
 bundle is strongly or weakly parabolic, it will be understood that it is strongly parabolic.
 
 Let $\cat{M}_{Higgs}(r, d, \alpha)$ be the moduli space of semi-stable (strongly) parabolic Higgs bundles
 of rank $r$, degree $d$, weight system $\alpha$  \cite{Y}, \cite{GL}.
 Then,
$\cat{M}_{Higgs}(r, d, \alpha)$ is an irreducible normal quasi-projective variety of dimension 
$2 r^2(g-1)+2 + m(r^2-r)$ (see \cite[p.n. 432]{GL}).
If $\alpha$ is a generic system of weights, then every 
semi-stable parabolic  Higgs bundle is stable, and the moduli space of stable Higgs  bundle lies in the smooth 
locus of $\cat{M}_{Higgs}(r, d, \alpha)$. As we have assumed $\alpha$ is generic system of weights, the moduli space  $\cat{M}_{Higgs}(r, d, \alpha)$  is smooth. 

Let $\cat{M}'_{Higgs}(r, d, \alpha) \subset \cat{M}_{Higgs}(r, d, \alpha)$ be the subset consists of those parabolic Higgs bundles whose underlying parabolic 
bundle is stable. As described above and from 
\eqref{eq:a21.2}
\begin{equation}
\label{eq:a22.1}
T^*\cat{M}(r, d, \alpha) \cong \cat{M}'_{Higgs}(r, d, \alpha).
\end{equation}

Let $Z := \cat{M}_{Higgs}(r, d, \alpha) \setminus \cat{M}'_{Higgs}(r, d, \alpha)$.
Then, in view of  \cite[Proposition 5.10]{BGL} for
$r \geq 3$,  we have 
\begin{equation}
\label{eq:a22.3}
\mbox{codim}(Z, \cat{M}_{Higgs}(r, d, \alpha) \geq 
2.
\end{equation}

Now, we recall the parabolic Hitchin map on the moduli space of strongly parabolic Higgs bundles \cite{GL}, \cite{SWW}.

Set
 $ (\Omega^1_X)^{i} S^{(i-1)} := (\Omega^1_X)^{\otimes i} \otimes \struct{X}(S)^{\otimes (i-1)},$ where $ i =1,\ldots, r,$ and consider the vector space 
 $$\cat{H}_P = \bigoplus_{i =1}^{r} \coh{0}{X}{(\Omega^1_X)^{i} S^{(i-1)}}.$$
 Then the dimension  $\mbox{dim}_{\C}{\cat{H}_P}$ is half of the dimension of the moduli space $\cat{M}_{Higgs}(r, d, \alpha)$ \cite[Section 3, p.n. 433]{GL}, that is, 
 $r^2(g-1) + 1+ \frac{1}{2}m(r^2-r)$.

 Define the parabolic Hitchin map 
 \begin{equation}
 \label{eq:a22.4}
 h_P : \cat{M}_{Higgs}(r, d, \alpha) \longrightarrow \cat{H}_P
 \end{equation}
 by sending each (strongly) parabolic Higgs bundle $(E_*, \Phi)$
 to the characteristic polynomial of $\Phi$.
 Then, the $h_P$ is a proper morphism \cite[Corollary 5.12]{Y}, and  for any generic point $a \in 
 \cat{H}_P$, the fibre $h_P^{-1}(a)$ is an abelian variety (see \cite[Lemma 3.2]{GL} and \cite[Theorem 6]{SWW}).

\section{Picard group of the moduli spaces}
\label{Pic}
The Picard group of moduli spaces is a very important 
invariant while studying the classification
problems in algebraic geometry. In this section, we compute the Picard group of the moduli spaces
$\cat{M}_{pc}(r, d, \alpha)$.

We first compute the dimension of the space of 
isomorphic stable parabolic connections on a parabolic 
vector bundle. Let $E_*$ be a full flag parabolic vector bundle 
over $X$ with a fixed $\alpha \in W^{(m)}_r(d)$.

Let $\cat{C}onn_{\alpha}(E_*)$ denote the space of all 
parabolic connections $D$ on $E_*$ such that $(E_*, D)$ is stable.  Notice that $\cat{C}onn_{\alpha}(E_*)$ is an affine space modelled over the vector space $\coh{0}
{X}{\Omega^1_X(S) \otimes \SParENd{E_*}}$.

Given a parabolic automorphism $\Phi$ of $E_*$ and a parabolic connection $D$ on $E_*$, the $\C$-linear morphism 
$\Phi \otimes \id{\Omega^1_X(S)} \circ D \circ 
\Phi^{-1}$ defines a parabolic connection on $E_*$.
In fact, $$(D,\Phi) \mapsto \Phi \otimes \id{\Omega^1_X(S)} \circ D \circ 
\Phi^{-1}$$ 
defines a natural action of $\text{Aut}(E_*)$ on $\cat{C}
onn_{\alpha}(E_*)$, called gauge transformation.
We would like to compute the dimension of the quotient 
space $\cat{C}onn_{\alpha}(E_*) / \text{Aut}(E_*)$, which 
parametrizes all isomorphic parabolic connections on $E_*
$. 

The Lie algebra of the holomorphic automorphism group 
$\text{Aut}(E_*)$ is $\coh{0}{X}{\ParENd{E_*}}$. Therefore, 
$$\text{dim}~\text{Aut}(E_*) = \text{dim}~\coh{0}{X}{\ParENd{E_*}}.$$

Choose any $D \in \cat{C}onn_{\alpha}(E_*)$. Since the pair $(E_*, D)$ is stable, from Lemma \ref{lem:1.a} \eqref{a2} the isotropy subgroup 
$$\text{Aut}(E_*)_{D} = \{\Phi \in \text{Aut}(E_*)~\vert~\Phi \otimes 
\id{\Omega^1_X(S)} \circ D \circ \Phi^{-1} = D \}$$ consists of  the scalar 
automorphisms of $E_*$.
Then, the dimension of the space $\cat{C}onn_{\alpha}(E_*) / \text{Aut}(E_*)$ is 
\begin{align*}
& \text{dim}~\coh{0}{X}{\Omega^1_X(S) \otimes \SParENd{E_*}} - 
\text{dim}~\coh{0}{X}{\ParENd{E_*}} \, + \, 1\\
 \, =  & \, \text{dim}~\coh{1}{X}{\ParENd{E_*}} - 
\text{dim}~\coh{0}{X}{\ParENd{E_*}} \,+ \, 1 
\end{align*}
\begin{equation}
\label{eq:dim}
\, =   \, - \chi(\ParENd{E_*}) + 1,
\end{equation}

where the first equality is due to Parabolic Serre duality and  $\chi (\ParENd{E_*})$ denotes the Euler-Poincar\'e characteristic of $\ParENd{E_*}$ over $X$.

Clearly $\ParENd{E_*}$ is a subsheaf of $\End{E}$. Then,
there is a natural skyscraper sheaf $\cat{K}_S$ supported on parabolic points $S$ such that 
\begin{equation*}
\label{eq:ses}
0 \to \ParENd{E_*} \to \End{E} \to \cat{K}_S \to 0
\end{equation*}
is a short exact sequence of sheaves on $X$. Then, 
we get 
\begin{equation}
\label{eq:add}
\chi (\End{E}) = \chi(\ParENd{E_*}) + \chi(\cat{K}_S)
\end{equation}
Since we are considering full flag filtrations, from 
\cite[Lemma 2.4]{BH}, we have 
\begin{equation}
\label{eq:sky}
\chi(\cat{K}_S) =  \frac{m (r^2-r)}{2}.
\end{equation}
 
Moreover, from Riemann-Roch Theorem, we have 
\begin{equation}
\label{eq:end}
\chi(\End{E}) = r^2(1-g).
\end{equation} 
 
Therefore, from \eqref{eq:dim}, \eqref{eq:add}, \eqref{eq:sky} and \eqref{eq:end}, we get that the dimension of the space $\cat{C}onn_{\alpha}(E_*) / \text{Aut}(E_*)$ is $$r^2(g-1)+1 + \frac{m(r^2-r)}{2}.$$

\begin{lemma}
\label{lem:a.2}
Let $E_*$ be a full flag stable parabolic vector bundle
of rank $r$ and $d$ with a fixed $\alpha \in W^{(m)}_r(d)$.  Then  $$\cat{C}onn_{\alpha}(E_*) / \Aut{E_*} = \text{affine space over the vector space}~ \coh{0}{X}{\Omega^1_X(S) \otimes \SParENd{E_*}},$$ and dimension of the space is equal to 
$r^2(g-1)+1 + \frac{m(r^2-r)}{2}$.
\end{lemma}

\begin{proof}
Since $E_*$ is stable parabolic bundle, we have 
$\coh{0}{X}{\ParENd{E_*}} = \C \cdot \id{E_*}$. Now, Lemma follows from above discussion.
\end{proof}

Let $$Z_{pc} \,:= \, \cat{M}_{pc}(r, d, \alpha) \setminus 
\cat{M}'_{pc}(r, d, \alpha).$$ Then, using the similar steps as in \cite[Lemma 3.1]{BM7}, we can show the following lemma. 

\begin{lemma}
\label{lemm:3} Let $\alpha \in W^{(m)}_r(d)$. Then,
for $r \geq 2$, and $g \geq 3$,  we have 
$$\codim{Z_{pc}, \cat{M}_{pc}(r, d, \alpha)} \geq 2.$$
\end{lemma}
\begin{proof}
Let $(E_*, D) \in Z_{pc}$. Then, $E_*$ is not stable.
Since $\alpha$ is generic, $E_*$ is not semi-stable.
Let 
$$ 0 = E^0_* \subset E^1_* \subset E^2_* \subset \cdots \subset E^{l-1}_* \subset E^l_* =  E_* $$
be the Harder-Narasimhan filtration of the parabolic vector bundle $E_*$.
The collection of pairs of integers 
$\{(\rk{E^i_*}, p\deg{E^i_*})\}_{i =1}^l$ is called the Harder-Narasimhan polygon of $E_*$ (see \cite[p.n.178]{SS}). 

Now, analogous to the techniques in \cite[Lemma 3.1]{BM7} (p.n. 303), the space of all isomorphism classes of parabolic  vector bundles over $X$ whose  Harder-Narasimhan polygon  coincides with the given  parabolic vector bundle $E_*$ is of dimension at most
$$r^2(g-1) - (r-1)(g-2) + \frac{m (r^2-r)}{2}.$$

We have already computed that  the dimension of the space of all isomorphic classes of parabolic connections , lying in $\cat{M}_{pc}(r, d, \alpha)$, on any given parabolic vector bundle 
$E'_*$ (assuming that parabolic connection exists on $E'_*$) which  is $$r^2(g-1)+1 + \frac{m(r^2-r)}{2}.$$ The subvariety of $\cat{M}_{pc}(r,d, \alpha)$ parametrizing all pairs of the form $(E'_*, D') \in \cat{M}_{pc}(r,d, \alpha)$ such that the Harder-Narasimhan polygon of the parabolic vector bundle $E'_*$ coincides with that of $E_*$
is at most of dimension
$$r^2(g-1) - (r-1)(g-2) + \frac{m (r^2-r)}{2} + r^2(g-1)+1 + \frac{m(r^2-r)}{2}$$
$$= 2 r^2 (g-1) - (r-1)(g-1) + 1 + m (r^2 -r).$$  

Since $\text{dim} \cat{M}_{pc}(r, d, \alpha) =  2r^2(g-1) +2 + m(r^2-r)$, we have 
$$\text{dim} \cat{M}_{pc}(r, d, \alpha) - [\,2 r^2 (g-1) - (r-1)(g-1) + 1 + m (r^2 -r)\,] = (r-1)(g-2)+1. $$
Thus, $\codim{Z_{pc}, \cat{M}_{pc}(r, d, \alpha)} \geq (r-1)(g-2) + 1.$
 \end{proof}

The morphism in \eqref{eq:a20.5} induces a morphism of 
 Picard groups
\begin{equation}
\label{eq:a25.1}
\iota^* : \Pic{\cat{M}_{pc}(r, d, \alpha)} \longrightarrow \Pic{\cat{M}'_{pc}(r, d, \alpha)}
\end{equation}
defined by restricting any line bundle $\eta$ over $\cat{M}_{pc}(r, d, \alpha)$ to $\cat{M}'_{pc}(r, d, \alpha)$.

Also, the morphism $\pi$ in \eqref{eq:a20.6} induces a morphism of Picard groups
\begin{equation}
\label{eq:a25.2}
\pi^* : \Pic{\cat{M}(r, d, \alpha)} \longrightarrow 
\Pic{\cat{M}'_{pc}(r, d, \alpha)}
\end{equation}
defined by pulling back of line bundles, that is, $\eta \mapsto \pi^* \eta$, where $\eta$ is a line bundle over 
$\cat{M}'_{pc}(r, d, \alpha)$.

\begin{theorem}
\label{thm:1}
 The morphisms defined in \eqref{eq:a25.1} and \eqref{eq:a25.2}
$$\Pic{\cat{M}(r, d, \alpha)} \xrightarrow{\pi^*} 
\Pic{\cat{M}'_{pc}(r, d, \alpha)} \xleftarrow{\iota^*} \Pic{\cat{M}_{pc}(r, d, \alpha)}$$
are isomorphisms.
\end{theorem}
 \begin{proof} 
  The morphism
$\iota^*$ defined in \eqref{eq:a25.1} is an isomorphism follows from the fact that 
$\codim{Z_{pc}, \cat{M}_{pc}(r, d, \alpha)} \geq 2$ as proved in Lemma \ref{lemm:3}.

Now, we show that $\pi^*$ is an isomorphism.
From Lemma \ref{lem:1}, the moduli space $\cat{M}'_{pc}(r, d, \alpha)$ is $T^*\cat{M}(r, d, \alpha)$-torsor.
We use this fact to compactify the moduli space $\cat{M}'_{pc}(r, d, \alpha)$, and using this compactification, we show that the Picard group 
$\Pic{\cat{M}'_{pc}(r, d, \alpha)}$ is isomorphic 
to $\Pic{\cat{M}(r, d, \alpha)}$.

We have seen that 
for any $E_* \in \cat{M}(r, d, \alpha)$, the fibre 
$\pi^{-1}(E_*)$ is an affine space modelled on
 $\coh{0}{X}{\Omega^1_X(S) \otimes \SParENd{E_*}}$.
 We know that the dual of an affine space is a vector space, therefore  the dual $$\pi^{-1}(E_*)^{\vee} = \{\varphi: \pi^{-1}
 (E_*) \to \C  ~\vert~ \varphi~ \mbox{is an affine linear map} \}$$  is a vector space over $\C$. 
 
 Let $$\psi: \cat{W} \to \cat{M}(r, d, \alpha)$$ be the algebraic vector bundle such that for every
 Zariski open subset $U$ of $\cat{M}(r, d, \alpha)$, a section of $\cat{W}$ over $U$ is an  
 algebraic function $f: \pi^{-1}(U) \to \C$ whose restriction to each fiber 
 $\pi^{-1}(E_*)$,  is an element of $\pi^{-1}(E_*)^{\vee}$.

 Thus, a fibre
 $\cat{W}(E_*) = \psi^{-1}(E_*)$ of
 $\cat{W}$ at $E_* \in \cat{M}(r, d, \alpha)$ is $\pi^{-1}(E_*)^{\vee}$.  The dimension of 
 $\pi^{-1}(E_*)^{\vee}$ is equal to  $\text{dim}_{\C} \pi^{-1}(E_*) + 1$, and since $E_*$ is stable, the dimension of $\pi^{-1}(E_*)^{\vee}$ is equal to
 $$r^2(g-1)+\frac{m(r^2-r)}{2} +2,$$ follows from Lemma \ref{lem:a.2}.
 
  Let $(E_*,D) \in \cat{M}'_{pc}(r, d, \alpha)$,
 and define a map $$\Psi_{(E_*, D)}: \pi^{-1}(E_*)^{\vee} \to \C,$$ by $\Psi_{(E_*, D)}(\varphi) = \varphi[(E_*, D)]$, that is in fact the evaluation map. Now,
 the kernel $\Ker{\Psi_{(E_*, D)}}$ defines a hyperplane in $\pi^{-1}(E_*)^{\vee}$
 denoted by $H_{(E_*, D)}$. 
 
 Let $\p (\cat{W})$ be the projective bundle defined by hyperplanes in the fibre $\pi^{-1}(E_*)^{\vee}$, that is, we have 
 \begin{equation}
 \label{eq:a26}
 \tilde{\psi}: \p(\cat{W}) \to \cat{M}(r, d, \alpha)
 \end{equation}
 induced from $\psi$. 

Next, define a map 
\begin{equation}
\label{eq:a27}
i: \cat{M}'_{pc}(r, d, \alpha)
  \to \p(\cat{W})
\end{equation} by sending  $(E_*, D)$ to the equivalence class of  $H_{(E_*, D)}$,
  which is an open embedding. Set 
\begin{equation}
\label{eq:a28}
{\bf H} = \p (\cat{W}) \setminus \cat{M}'_{pc}(r,d,\alpha).
\end{equation}  
  Then $\tilde{\psi}^{-1}(E_*) \cap {\bf H} $ is a projective hyperplane in $\tilde{\psi}^{-1}(E_*)$ for every $E_* \in \cat{M}(r, d, \alpha)$, and hence ${\bf H}$ is a hyperplane at infinity.

 We first 
show that $\pi^*$ is injective.
Let $\eta \to 
\cat{M}(r, d, \alpha)$ be a line bundle such that $\pi^* \eta$ is a trivial line bundle
over $\cat{M}'_{pc}(r, d, \alpha)$. A trivialization of $\pi^* \eta$ is equivalent to have 
 a nowhere vanishing section of $\pi^* \eta$ over $\cat{M}'_{pc}(r, d, \alpha)$. Fix 
$s \in \coh{0}{\cat{M}'_{pc}(r, d, \alpha)}{\pi^* \eta}$ a nowhere vanishing section. Choose a
point $E_* \in \cat{M}(r, d, \alpha)$.  Then, from the following commutative diagram 

\begin{equation}
\label{eq:a29}
\xymatrix{
\pi^* \eta \ar[d] \ar[r]^{\tilde{\pi}} & \eta \ar[d] \\
\cat{M}'_{pc}(r, d, \alpha) \ar[r]^{\pi} & \cat{M}(r, d, \alpha)\\
}
\end{equation}
we get 
\begin{equation*}
\label{eq:a30}
s|_{\pi^{-1}(E_*)}: \pi^{-1}(E_*) \to \eta(E_*)
\end{equation*}
 a nowhere vanishing map. Notice that $\pi^{-1}(E_*) \cong \C^N$ and $\eta(E_*) \cong \C$, where $N = r^2(g-1)+\frac{m(r^2-r)}{2} + 1$. Now, any nowhere vanishing algebraic function on an affine space
$\C^N$ is a constant function, that is, $s|_{\pi^{-1}(E_*)}$ is a constant
function and hence corresponds to a non-zero vector $\alpha_{E_*} \in \eta(E_*)$.
Since $s$ is constant on each fiber of $\pi$,  the trivialization $s$ of $\pi^*\eta$ descends to a trivialization of the
line bundle $\eta$ over $\cat{M}(r, d, \alpha)$, and hence giving a nowhere vanishing
section of $\eta$ over $\cat{M}(r, d, \alpha)$.  Thus,  $\eta$ is a trivial line bundle
over $\cat{M}(r, d, \alpha)$.

 It remains to show that $\pi^*$ is surjective. Let 
 $\vartheta \to \cat{M}'_{pc}(r, d, \alpha)$ be an algebraic line 
 bundle.  Since $\p (\cat{W})$ is a smooth compactification of $\cat{M}'_{pc}(r, d, \alpha)$ follows from the embedding $i : \cat{M}'_{pc}(r, d, \alpha) \hookrightarrow
 \p (\cat{W}) $ in \eqref{eq:a27}, the homomorphism of Picard groups 
 $$i^* : \Pic{ \p (\cat{W})} \to \Pic{\cat{M}'_{pc}(r, d, \alpha)}$$ defined by pull back of line bundles via $i$ in \eqref{eq:a27}, is a surjective homomorphism. Therefore, 
 we can extend $\vartheta$ to a line bundle
 $\vartheta'$ over $\p (\cat{W})$.
 Again from the morphism $\tilde{\psi}: \p(\cat{W}) \to \cat{M}(r, d, \alpha)$ in \eqref{eq:a26}  and from \cite{H}, Chapter III, Exercise 12.5, p.n. 291,
we have
\begin{equation}
\label{eq:a31}
 \Pic {\p (\cat{W})} \cong \tilde{ \psi}^*\Pic{\cat{M}(r, d, \alpha)}\oplus  \Z \struct{\p (\cat{W})}(1).
\end{equation}
Therefore,
\begin{equation}
\label{eq:a32}
\vartheta' = \tilde{\psi}^* \Lambda \otimes \struct{\p (\cat{W})}(l)
\end{equation}
where $\Lambda$ is a line bundle over $\cat{M}$ and 
$l \in \Z$.
Since ${\bf H} = \p (\cat{W}) \setminus \cat{M}'_{pc}(r, d, \alpha)$ is the 
hyperplane at infinity, using \eqref{eq:a31} the line bundle 
$\struct{\p (\cat{W})}({\bf H})$ associated to the divisor ${\bf H}$ can be expressed
as 
\begin{equation}
\label{eq:a33}
\struct{\p (\cat{W})}({\bf H}) = \tilde{\psi}^* \Gamma \otimes \struct{\p (\cat{W})}(1)
\end{equation}
for some line bundle $\Gamma$ over $\cat{M}(r, d, \alpha)$.
Now, from \eqref{eq:a32} and \eqref{eq:a33}, we get
\begin{equation*}
\label{eq:a34}
\vartheta' = \tilde{\psi}^*(\Lambda \otimes 
(\Gamma^{\vee})^{\otimes l}) \otimes \struct{\p (\cat{W})}(l{\bf H}).
\end{equation*}
Since, the restriction of the line bundle $\struct{\p 
(\cat{W})}({\bf H})$ to the complement $$\p (\cat{W}) \setminus {\bf H} = 
\cat{M}'_{pc}(r, d, \alpha)$$ is the trivial line bundle and restriction of $\tilde{\psi}$ to $\cat{M}'_{pc}(r, d, \alpha)$ is the map $\pi$ defined in \eqref{eq:a21.1}, therefore, we have
\begin{equation*}
\label{eq:a35}
\vartheta =  \pi^*(\Lambda \otimes (\Gamma^{\vee})^{\otimes l}).
\end{equation*}
This completes the proof of the theorem.

\end{proof}

\begin{remark} \mbox{}
\label{rmk:1} There is an alternative way of proving  that $\pi^*$ as defined in \eqref{eq:a25.2} is an 
isomorphism, without using the compactification 
$\p (\cat{W})$ of $\cat{M}'_{pc}(r, d, \alpha)$.
Recall that if we have an algebraic morphism 
 between varieties $f : \cat{X} \longrightarrow \cat{S}$
 such that the fibres are geometrically connected, then 
 there exits an exact sequence
 \begin{equation}
 \label{eq:b23}
 1 \to \Pic{\cat{S}} \to \Pic{\cat{X}} \to \text{Pic}_{\cat{X}/\cat{S}}(\cat{S}) \to \text{Br}(\cat{S}),
 \end{equation}
 of groups, where $\text{Pic}_{\cat{X}/\cat{S}}(\cat{S})$ denote the group of section of relative Picard scheme 
 $\text{Pic}_{\cat{X/\cat{S}}}$ over $S$, and 
 $\text{Br}(\cat{S})$ denote the Brauer group of $\cat{S}$.
 In our set up, we have following exact sequence
 corresponding to the morphism $\pi$  defined in \eqref{eq:a21.1}
 \begin{equation}
 \label{eq:b24}
 1 \to \Pic{\cat{M}(r, d, \alpha)} \to \Pic{\cat{M}'_{pc}(r, d, \alpha)}
 \to \text{Pic}_{\cat{M}'_{pc}(r, d, \alpha)/ \cat{M}(r, d, \alpha
 )}(\cat{M}(r, d, \alpha))
 \end{equation}
 
 Since, each fibre of the morphism $\pi$ (see \eqref{eq:a21.1}) is an affine space, the group 
 $\text{Pic}_{\cat{M}'_{pc}(r, d, \alpha)/ \cat{M}(r, d, \alpha
 )}(\cat{M}(r, d, \alpha))$ is trivial. 

\end{remark}

Similarly, we can compactify (see Proposition \ref{prop:3}) the moduli space $\cat{M}'_{pc}(r, \alpha, \xi)$ and  we can show the following.
\begin{proposition}
\label{prop:1}
$\Pic{\cat{M}_{pc}(r, \alpha, \xi)} \cong \Pic{\cat{M}(r,
 \alpha, \xi)}.$
\end{proposition}

\section{Space of holomorphic connections on an ample line bundle}
\label{Space_hol_conn}
 In this section, we assume that $r \geq 3$, and  consider the space of holomorphic 
 connections on an ample line bundle $L$ over the moduli space $\cat{M}(r, d, \alpha)$.
 For the simplicity of the notation, we just write 
 $\cat{M}$ for the moduli space $\cat{M}(r, d, \alpha)$.

 Let $L$ be an ample holomorphic vector bundle over 
 $\cat{M}$.
 Let $\Omega^1_{\cat{M}}$ denote the sheaf of holomorphic $1$-forms on $\cat{M}$. Consider the 
 space $\cat{C}(L)$ of holomorphic connections on 
 $L$, that is, for every analytic open subset $U$ of 
 $\cat{M}$, $\cat{C}(L) \vert_{U}$
 consists of following operators
 $$\nabla : L \vert_{U} \longrightarrow L \vert_{U} \otimes \Omega^1_{\cat{M}} \vert_{U}$$
 satisfying the Leibniz rule
 $$\nabla(fs) =f \nabla (s) + \text{d}f \otimes s,$$
 where $f$ is a holomorphic function on $U \subset \cat{M}$, and $s$ is a holomorphic section of $L$ over $U$.
 Then, there is a natural projection morphism 
 
 \begin{equation}
 \label{eq:a36}
 \Psi : \cat{C}(L) \longrightarrow \cat{M} :=\cat{M}(r, d, \alpha).
 \end{equation}

 For any analytic open subset $U$ of 
 $\cat{M}$, $\Psi^{-1}(U)$  is an affine space modelled over the vector space
 $\coh{0}{U}{\Omega^1_{\cat{M}}}.$
 Thus, the space $\cat{C}(L)$ is a $T^* \cat{M} $-torsor.  Now, using the similar technique
 as above we can show the following.
 
 \begin{proposition}
 \label{prop:2}
 $\Pic{\cat{C}(L)} \cong \Pic{\cat{M}}.$
 \end{proposition}
 
 Next, we want to study the space of regular function on 
 $\cat{C}(L)$.
 We recall the definition of differential operators of finite order on $L$. 
 Let $k \geq 0$ be an integer.
 A differential operator of order $k$ on $L$ is a $\C$-linear 
map
\begin{equation}
\label{eq:a36.1}
Q: L \to L
\end{equation}
such that for every open subset $U$ of $\cat{M}$ and for every 
$f \in \struct{\cat{M}}(U)$, the bracket 
$$[Q|_U,f]:L|_U \to L|_U$$ defined as 
\begin{equation*}
\label{eq:a37}
[Q|_U,f]_V(s) = Q_V(f|_V s) - f|_V Q_V(s)
\end{equation*}
is a differential operator of order $k-1$, for every open subset $V$ of 
$U$, and for all $ s \in L(V)$,
where differential operator of order zero from $L$ to $L$ is just 
$\struct{\cat{M}}$-module homomorphism.

Let $\Diff[k]{\cat{M}}{L}{L}$ denote the set of all differential operators of order 
$k$.  Then 
$\Diff[k]{\cat{M}}{L}{L}$ is an 
$\struct{\cat{M}}(\cat{M})$-module. For every open subset $U$ of $\cat{M}$, 
$$U \mapsto \Diff[k]{\cat{M}}{L|_U}{L|_U}$$ is a sheaf of differential operator of 
order $k$ from $L|_U$ to 
$L|_U$. This sheaf is denoted by $\DIFF[k]{\cat{M}}{L}{L}$, which is locally free.
For $k \geq 0$, we denote by $\cat{D}^k(L)$ the vector bundle over $\cat{M}$ 
defined by the sheaf $\DIFF[k]{\cat{M}}{L}{L}$. Note that $\cat{D}^0(L) = 
\struct{\cat{M}}$ and we have following inclusion of vector bundles
\begin{equation}
\label{eq:a38}
\struct{\cat{M}} = \cat{D}^0(L) \subset \cdots \subset \cat{D}^k(L) \subset \cat{D}^{k+1}(L) \subset \cdots
\end{equation}                                                                        

Let $Q$ be a first order differential operator on $L$.
Define a map 
\begin{equation}
\label{eq:a39}
\sigma(Q) : \struct{\cat{M}} \longrightarrow \struct{\cat{M}} = \END{L}
\end{equation}
by $$\sigma(Q)(f) = [Q\vert_{U}, f]$$ for every open 
subset $U \subset \cat{M}$, and 
$f \in \struct{\cat{M}}(U)$. Then 
$\sigma(Q)$ is a $\C$-derivation, that is, $\sigma(Q)$
gives a section of $T\cat{M}$.
Thus, we get a short exact sequence of vector bundles,
called the {\bf Atiyah exact sequence} \cite{A}
\begin{equation}
\label{eq:a40}
0 \to \struct{M} \xrightarrow{\iota} \At{L} := \cat{D}^{1}(L) \xrightarrow{\sigma} T\cat{M} \to 0,
\end{equation}
where $\At{L}$ is called the Atiyah bundle of $L$.

 \begin{theorem}
 \label{thm:2} 
 Let $L$ be an ample line bundle over $\cat{M}$, and 
 $\struct{\cat{C}(L)}$ denote the sheaf of regular functions on $\cat{C}(L)$. Then
 \begin{equation}
 \label{eq:40.1}
 \coh{0}{\cat{C}(L)}{\struct{\cat{C}(L)}} = \C
 \end{equation}

 \end{theorem}
 \begin{proof} 
 Consider the dual of the  Atiyah exact sequence in \eqref{eq:a40}, that is, 
 \begin{equation}
\label{eq:a41}
0 \to T^*\cat{M} 
\xrightarrow{\sigma^*} 
\At{L}^* \xrightarrow{\imath^*}   \struct{\cat{M}} \to 0.
\end{equation}

Consider $\struct{\cat{M}}$ as the trivial line 
bundle $\cat{M} \times \C$. Let $$ \alpha : \cat{M} \to \cat{M} \times \C$$ be a
holomorphic section of the trivial line bundle defined by $x \mapsto (x,1)$.
Let $Y = \Img{\alpha} \subset \cat{M} \times \C$ be the image of $\alpha$. Then 
$$\cat{C}(L) = {\imath^*}^{-1}Y \subset \At{L}^*.$$

In fact, for every open subset $U \subset \cat{M}
$, a holomorphic section of $\cat{C}(L)|_{U}$ over $U$ gives a holomorphic
splitting of the Atiyah exact sequence \eqref{eq:a40}, 
associated to the holomorphic vector bundle $L|
_{U} \to U$. 
 For instance, suppose $\tau: U \to \cat{C}(L)|
_{U}$ is a holomorphic section. Then $\tau$ will be a 
holomorphic section of $\At{L}^*|_{U}$  over $U$, 
because $ \cat{C}(L) =
{\imath^*}^{-1}Y \subset \At{L}^*$. Since $$ 
 \tau \circ \imath =  \imath^*(\tau) = \id{U},$$ so we 
 get a holomorphic  splitting $\tau$ of the Atiyah 
 exact sequence \eqref{eq:a40}
 associated to $L|_{U}.$ Thus, $L|_{U}$ admits 
 a holomorphic connection. 
 Conversely, given any splitting of Atiyah exact sequence 
 \eqref{eq:a40} over an open subset $U \subset \cat{M}$, we get a holomorphic section of $\cat{C}(L)|_{U}$ over $U$.
 
Let 
 $\p(\At{L})$ be the projectivization of 
$\At{L}$, that is, $\p(\At{L})$ parametrises hyperplanes in $\At{L}$.
Let $\p({T \cat{M}})$ be the projectivization of the tangent bundle $T\cat{M}$. 
Notice that $\p(T\cat{M})$ is a subvariety of 
$\p (\At{L})$, and  $\p(T\cat{M})$ is the zero
locus of the of a section of the tautological line bundle
$\struct{\p (\At{L})}(1)$. Now, observe that
$$\cat{C}(L) = \p(\At{L}) \setminus \p(T \cat{M}).$$ 
So we have 
\begin{equation}
\label{eq:a42}
\coh{0}{\cat{C}(L)}{\struct{\cat{C}(L)}} =
\varinjlim_{k \geq 0} \coh{0}{\p (\At{L})}{\struct{\p \At{L}}(k)}.
\end{equation}

Since for any finite dimensional vector space $V$ over 
$\C$ and for every $k \geq 0$, we have 
$\coh{0}{\p(V)}{\struct{\p(V)}(k)} = \cat{S}ym^k(V)$, 
where $\cat{S}ym^k(V)$ denote the $k$-th symmetric powers of $V$. We get a natural isomorphism 
$$\coh{0}{\p (\At{L})}{ \struct{\p \At{L}}(k)} \cong \coh{0}{\cat{M}}{ \cat{S}ym^k\At{L}},$$
where $\cat{S}ym^k\At{L}$ denote the $k$-the symmetric powers of $\At{L}$,
and hence from \eqref{eq:a42}, we get
\begin{equation}
\label{eq:a43}
\coh{0}{\cat{C}(L)}{\struct{\cat{C}(L)}} = \varinjlim_{k} \coh{0}{\cat{M}}{ \cat{S}ym^k\At{L}}.
\end{equation}

The symbol operator 
\begin{equation*}
\label{eq:a44}
\sigma:\At{L} \to   T\cat{M}
\end{equation*}
as described  in  \eqref{eq:a39},
 induces a morphism 
\begin{equation*}
\label{eq:a45}
\cat{S}ym^k(\sigma): \cat{S}ym^k \At{L} \to 
\cat{S}ym^k T \cat{M}
\end{equation*}
on $k$-th symmetric powers of bundles.
In view of the following composition 
\begin{equation*}
\cat{S}ym^{k-1} \At{L} = \struct{\cat{M}} \otimes \cat{S}ym^{k-1} \At{L} \hookrightarrow 
\At{L} \otimes \cat{S}ym^{k-1} \At{L} \to 
\cat{S}ym^{k} \At{L},
\end{equation*}
we have 
\begin{equation*}
\label{eq:a46}
\cat{S}ym^{k-1} \At{L} \subset \cat{S}ym^k \At{L}
~~~ \mbox{for all}~ k \geq 1.
\end{equation*}
In fact, we have the symbol exact sequence associated to 
$L$ over
$\cat{M}$,
\begin{equation}
\label{eq:a47}
0 \to \cat{S}ym^{k-1} \At{L} \to \cat{S}ym^{k} \At{L} \xrightarrow{\cat{S}ym^k (\sigma)} \cat{S}ym^k T \cat{M} \to 0.
\end{equation}

In other words, we get a filtration 
\begin{equation*}
\label{eq:32}
0 \subset \cat{S}ym^0 \At{L} \subset \cat{S}ym^1 
\At{L} \subset \ldots \subset \cat{S}ym^{k-1} 
\At{L} \subset \cat{S}ym^k \At{L} \subset 
\ldots
\end{equation*}

 such that 
 \begin{equation}
 \label{eq:33}
 \cat{S}ym^{k} \At{L} / \cat{S}ym^{k-1} \At{L} 
 \cong \cat{S}ym^k T \cat{M}~~~ \mbox{for all}~ k 
 \geq 1.
 \end{equation}

To prove \eqref{eq:40.1}, it is enough to show that 
\begin{equation}
\label{eq:35}
\coh{0}{\cat{M}}{\cat{S}ym^{k-1}\At{L}}
\cong \coh{0}{\cat{M}}{\cat{S}ym^k\At{L}}
~~~ \mbox{for all}~ k \geq 1.
\end{equation}
%
We have the following commutative diagram
\begin{equation}
\label{eq:cd1}
\xymatrix@C=2em{
0 \ar[r] & \cat{S}ym^{k-1} \At{L} \ar[d] \ar[r] & \cat{S}ym
^k \At{L} 
\ar[d] \ar[r]^{\sigma_k} &  \cat{S}ym^k T \cat{M} \ar[d] \ar[r] & 0 \\
0 \ar[r] &  \cat{S}ym^{k-1} T \cat{M} \ar[r] & 
\frac{\cat{S}ym^k \At{L}}{\cat{S}ym^{k-2} \At{L}} \ar[r] &  \cat{S}ym^k T \cat{M} \ar[r] & 0 
}
\end{equation}

which gives the following commutative 
diagram of long exact sequences
\begin{equation}
\label{eq:cd2}
\xymatrix@C=2em{
\cdots \ar[r] & \coh{0}{\cat{M}}{\cat{S}ym^k T 
\cat{M}} \ar[d] \ar[r]^{\delta'_{k}} & \coh{1}
{\cat{M}}{\cat{S}ym^{k-1}\At{L}} \ar[d] 
\ar[r] & \cdots  \\
\cdots \ar[r] & \coh{0}{\cat{M}}{\cat{S}ym^k T 
\cat{M}}       \ar[r]^{\delta_{k}} & \coh{1}
{\cat{M}}{ \cat{S}^{k-1} T \cat{M}} 
\ar[r] & \cdots }
\end{equation}
In order to show \eqref{eq:40.1}, in view of \eqref{eq:a43}, it is enough to prove that 
the boundary operator $\delta'_k$ is injective for all $k
\geq 1$ and   which is equivalent to showing that 
the boundary operator 
\begin{equation}
\delta_{k}: \coh{0}{\cat{M}}{  \cat{S}ym^k T \cat{M}}      \to  \coh{1}{\cat{M}}
{ \cat{S}ym^{k-1} T \cat{M}}
\end{equation}
is injective for every $k \geq 1$.

Further a connecting homomorphism can be expressed as the
cup product by the extension class of the corresponding
short exact  sequence.  Denote the extension class
of the following short exact sequence 
\begin{equation}
\label{eq:a49.1}
0 \to  \cat{S}ym^{k-1}(T \cat{M}) \to 
\frac{\cat{S}ym^k\At{L}}{\cat{S}ym^{k-2}\At{L}} \to \cat{S}ym^k(T \cat{M}) \to 0 
\end{equation}
by $\gamma_{k}$.
   
Moreover, $\gamma_{k}$ can be expressed in terms of the
first Chern class $c_1(L)$, because
the first Chern class of $c_1(L)$ is nothing  but
the extension class of the following {\bf Atiyah exact
sequence} (see \cite{A})

\begin{equation}
\label{eq:a53}
0 \to \struct{\cat{M}} \to \At{L} \xrightarrow{ \sigma} T \cat{M} \to 0,
\end{equation}
and the short exact sequence \eqref{eq:a47} is 
the $k$-th symmetric power of the Atiyah exact sequence 
\eqref{eq:a53}.
 
Thus, the connecting homomorphism $\delta_{k}$ can be described using the first Chern
class $c_1(L) \in \coh{1}{\cat{M}}
{ T^* \cat{M}}$ of the line bundle $L$. 

The cup product with $ kc_1(L)$ gives rise to a homomorphism
\begin{equation*}
\label{eq:a54}
\mu: \coh{0}{\cat{\cat{M}}}{ \cat{S}ym^k T \cat{M}} \to \coh{1}{\cat{M}}{\cat{S}ym^k T \cat{M} \otimes T^* \cat{M}}
\end{equation*}
Also, we have a canonical homomorphism of vector bundles
\begin{equation*}
\label{eq:a55}
\beta: \cat{S}ym^k T\cat{M} \otimes T^*\cat{M} \to  \cat{S}ym^{k-1}T \cat{M}
\end{equation*}
which induces a morphism of  \C-vector spaces
\begin{equation*}
\label{eq:a56}
\beta^*:\coh{1}{\cat{M}}{\cat{S}ym^k T \cat{M}
 \otimes T^* \cat{M}} \to \coh{1}
{\cat{M}}{\cat{S}ym^{k-1} T \cat{M}}.
\end{equation*}
So, we get a morphism
\begin{equation*}
\label{eq:a57}
\tilde{\mu} = \beta^* \circ \mu: \coh{0}{\cat{M}}{\cat{S}ym^k T 
\cat{M}} \to \coh{1}{\cat{M}}
{\cat{S}ym^{k-1} T \cat{M}}.
\end{equation*}
Then $\tilde{\mu} = \delta_{k}$.  Now, it 
is enough to show that $\tilde{\mu}$ is injective.

Consider the natural projection 
\begin{equation*}
\label{eq:a49}
p : T^*\cat{M} \longrightarrow \cat{M},
\end{equation*}
where $T^*\cat{M}$ denote the cotangent bundle of $\cat{M}$. Then, we have 
\begin{equation}
\label{eq:a50}
 \coh{0}{T^*\cat{M}}{\struct{T^*\cat{M}}} = \coh{0}{\cat{M}}{p_* \struct{T^*\cat{M}}}.
\end{equation}
Moreover, using the projection formula we have 
\begin{equation}
\label{eq:a51}
p_* \struct{T^*\cat{M}} = \bigoplus_{k \geq 0}  \cat{S}ym^{k} T\cat{M}.
\end{equation}
Using \eqref{eq:a50} and \eqref{eq:a51} we get
\begin{equation}
\label{eq:a52}
\coh{0}{T^*\cat{M}}{\struct{T^*\cat{M}}} = \bigoplus_{k \geq 0} \coh{0}{\cat{M}}{ \cat{S}ym^{k} T \cat{M}}
\end{equation}

Now, to compute $\coh{0}{T^*\cat{M}}{\struct{T^*\cat{M}}}$
we use the Hitchin fibration for the moduli space of 
strongly parabolic Higgs bundles as defined in \eqref{eq:a22.4}. Recall that for
any generic point $a \in \cat{H}_P$, we have  $h_P^{-1}(a) = A $, where
$A$ is some abelian variety (for more details see \cite{SWW}, \cite{BGL}), and we will be using this fact 
showing that $\tilde{\mu}$ is injective.

Let $g: T^*\cat{M} \to \C$ be an algebraic function. Since $T^*\cat{M}$ is an open subset of $\cat{M}_{Higgs }(r, d, \alpha)$ such that the complement has codimension  at least $2$ (see \eqref{eq:a22.3}), from
Hartog's theorem, the algebraic function $g$ is extended to an algebraic function $$\tilde{g} : \cat{M}_{Higgs}(r, d, \alpha) \to \C.$$  Then its restriction  $\tilde{g}|_{h_P^{-1}(a)}: h_P^{-1}(a) \to \C$ to $h_P^{-1}(a)$ is an algebraic function. Since $h_P^{-1}(a)$ is an abelian variety,  $\tilde{g}$ is a constant 
function. Therefore, on generic fibre $h^{-1}(a)$, $\tilde{g}$ is constant, and $h_P$ is proper, hence gives an algebraic function on
$\cat{H}_P$.    Thus, any algebraic function on $T^*\cat{M}$ descends to an algebraic function on $\cat{H}_P$.

Set $\cat{P} = \mbox{d}( \coh{0}{\cat{H}_P}
{\struct{\cat{H}_P}}) \subset  \coh{0}{\cat{H}_P}
{\Omega^1_{\cat{H}_P}}$ the space of all exact algebraic $1$-form.
Define a map 
\begin{equation}
\label{eq:a53.1}
\theta: \coh{0}{T^*\cat{M}}
{\struct{ T^* \cat{M}}} \to \cat{P}
\end{equation}
by $g \mapsto d \tilde{g}$, where $\tilde{g}$ is the function which is defined by descent of $g$.  Then $\theta$ is an isomorphism.

From \eqref{eq:a52} and \eqref{eq:a53.1}, we have
\begin{equation}
\label{eq:a54.2}
\theta:\bigoplus_{k \geq 0}
\coh{0}{\cat{M}}
{\cat{S}ym^{k} T \cat{M}} \to \cat{P}
\end{equation}
which is an isomorphism.
Let us denote the restriction of $h_P$ on $T^*\cat{M}$ by  $h'_{P}$.

Let $T_{h'_P} = T_{T^* \cat{M} / \cat{H}_P} = \SKer{\mbox{d}h'_P}$
be the relative tangent sheaf on $T^*\cat{M}$,
where $$\mbox{d}h'_P: T(T^*\cat{M}) \to {h'}_P^*T \cat{H}_P$$ morphism 
of bundles.

 Note that 
$ \coh{0}{\cat{H}_P}{\Omega^1_{\cat{H}_P}} \subset  \coh{0}{T^*\cat{M}}
{T_{h'_P}} $, and hence from \eqref{eq:a54.2}, we have
an injective homomorphism
\begin{equation*}
\label{eq:a55.1}
\nu: \cat{P} = \bigoplus_{k \geq 0} \theta(
\coh{0}{\cat{M}}
{\cat{S}ym^{k} T \cat{M}}) \to \coh{0}{T^*\cat{M}}{T_{h'_P}}.
\end{equation*}

Consider the morphism 

$$\coh{0}{T^*\cat{M}}{T_{h'_P}} \to \coh{1}{T^*\cat{M}}{T_{h'_P} \otimes T^* T^* \cat{M}}$$ defined 
by taking cup product with the first Chern class $c_1(p^* L) \in  \coh{1}{T^*\cat{M}}{ T^* T^* \cat{M}}$.

Using the pairing $$T_{h'_P} \otimes T^* T^* \cat{M} \to \struct{T^* \cat{M}},$$ we get a homomorphism
\begin{equation*}
\label{eq:a56.1}
\eta: \coh{0}{T^*\cat{M}}{T_{h'_P}} \to \coh{1}{T^*\cat{M}}{\struct{T^*\cat{M}}}
\end{equation*}

Since $c_1(p^* L) = p^*c_1 (L)$, we have
\begin{equation}
\label{eq:a57.1}
k \eta \circ \nu \circ \theta(\omega_k) = \tilde{\mu}(\omega_k),
\end{equation}
for all $\omega_k \in \coh{0}{\cat{M}}
{\cat{S}ym^{k} T \cat{M}})$.
Since $\nu$ and $\theta$ are injective homomorphisms, it is enough to 
show that $\eta|_{\nu(\cat{P})}$ is injective homomorphism.

Let $\omega \in \cat{P} \setminus \{0\}$ be a non-zero
exact $1$-form. Choose $a \in \cat{H}_P$ such that $\omega(a) \neq 0$. 
In view of \eqref{eq:a22.1} and \eqref{eq:a22.3}, the generic fibre has the form
 $${h'_P}^{-1}(a) = A \setminus F,$$ where $A$ is an abelian variety and $F$ is a subvariety of $A$ such that $\mbox{codim}(F,A) \geq 2$.
Now, $\eta(\nu(\omega)) \in \coh{1}{T^*\cat{M}}{\struct{T^*\cat{M}}}$ and we have restriction map $$\coh{1}{T^*\cat{M}}{\struct{T^*\cat{M}}} \to \coh{1}{{h'_P}^{-1}(a)}{\struct{{h'_P}^{-1}(a)}}.$$
Since $ \omega(a) \neq 0$, $\eta(\nu(\omega)) \in \coh{1}{{h'_P}^{-1}(a)}{\struct{{h'_P}^{-1}(a)}}$.
Because of the following isomorphisms
\begin{equation*}
\label{eq:a58}
 \coh{1}{{h'_P}^{-1}(a)}{\struct{{h'_P}^{-1}(a)}} \cong  \coh{1}{A}{\struct{A}} \cong  \coh{0}{A}{TA},
\end{equation*}
it follows that $\eta(\nu(\omega)) \neq 0$.
This completes the proof. 
\end{proof}

Similar steps  involved in the above Theorem \ref{thm:2}, and using \eqref{eq:a38}, we have generalisation of 
\cite[Corollary 2.3]{AS1} in parabolic set up. 

\begin{corollary}
\label{cor:1}
Let $L$ be an ample line bundle over $\cat{M}(r, d, \alpha)$. Then,
\begin{equation*}
\label{eq:dif}
\coh{0}{\cat{M}}{\cat{D}^k(L)} = \C
\end{equation*}
for every $k \geq 0$.
\end{corollary}

\begin{remark} \mbox{}
\label{rmk:2}
\begin{enumerate}
\item Note that the above 
 Theorem \ref{thm:2} is equivalent to proving 
 $$\coh{0}{\cat{M}}{\cat{S}ym^k\At{L}} = \C$$
 for every $k \geq 0$.

\item Since $\cat{M}'_{pc}(r, d, \alpha)$ and $\cat{C}(L)$ are 
$T^*\cat{M}(r, d, \alpha)$-torsors,  there is a natural 
question to ask,
\begin{question}
\label{q:1}
Does there exist an ample line bundle $L$ over $\cat{M}(r, d, \alpha)$ and an isomorphism 
$$\varpi : \cat{M}'_{pc}(r, d, \alpha) \longrightarrow \cat{C}(L)$$ of varieties 
such that the following diagram 
\begin{equation*}
\xymatrix{
\cat{M}'_{pc}(r, d, \alpha) \ar[r]^{\varpi} \ar[d]^{\pi} & \cat{C}(L) \ar[dl]^{\Psi} \\
\cat{M}(r, d, \alpha) \\
}
\end{equation*}
commutes ?
\end{question}

If the answer of the Question \ref{q:1} is in affirmative, then from the Theorem \ref{thm:2}, the moduli space $\cat{M}_{pc}(r, d, \alpha)$ does not admit 
any non-constant algebaric function.

\end{enumerate}
\end{remark}

\section{Divisor at infinity}
In \cite{BR}, compactification of the moduli space of 
logarithmic connections has been described, and authors have studied the some important properties, like numerically effectiveness,  of the smooth 
divisor at infinity.

From the proof of the Theorem \ref{thm:1} in section 
\ref{Pic}, we have a natural compactification 
$\p (\cat{W})$ of the moduli space $\cat{M}'_{pc}(r,d,
\alpha)$ such that the complement $\p (\cat{W}) 
\setminus \cat{M}'_{pc}(r,d,\alpha)$ is a smooth divisor 
${\bf H}$ (see \eqref{eq:a28}) at infinity.

In present section, we compactify the moduli space 
$\cat{M}'_{pc}(r,\alpha, \xi)$ as described in 
\eqref{eq:a21}, and study the properties of the smooth 
divisor at infinity.

\begin{proposition}
\label{prop:3}
There exists an algebraic vector bundle 
\begin{equation}
\label{eq:a59}
\Psi :\cat{V} \longrightarrow \cat{M}(r,\alpha, \xi)
\end{equation}
such that 
$\cat{M}'_{pc}(r, \alpha, \xi)$ is embedded in $\p(\cat{V})$ with $$
{\bf H_0} = 
\p (\cat{V}) \setminus \cat{M}'_{pc}(r, \alpha, \xi)$$ as a smooth divisor at infinity.
\end{proposition}
\begin{proof}
Recall that the map $\pi_{\xi}$ defined in \eqref{eq:a21.1} is a $T^*\cat{M}(r, \alpha, \xi)$-torsor.  
Let 
\begin{equation*}
\label{eq:51}
\Psi: \cat{V} \to \cat{M}(r, \alpha, \xi)
\end{equation*}
 be the algebraic vector bundle such that for every
 Zariski open subset $U$ of $\cat{M}(r, \alpha, \xi)$, a section of $\cat{V}$ over $U$ is an  
 algebraic function $f: \pi_{\xi}^{-1}(U) \to \C$ whose restriction to each fiber 
 $\pi_{\xi}^{-1}(E_*)$,  is an element of $\pi_{\xi}^{-1}(E_*)^{\vee}$. 
 Now, the rest of the proof is exactly similar to a part of the proof of the Theorem \ref{thm:1}.
 \end{proof}

Now, we generalise 
\cite[proposition 5.1]{BR} in parabolic set up.

\begin{proposition}
\label{prop:4}
Consider the smooth divisor $\bf{H}_0$ defined in Proposition \ref{prop:3}. Then,
the divisor $\bf{H}_0$ is numerically effective if and only if the tangent bundle $T \cat{M}(r, \alpha, \xi)$ is numerically effective. 

\end{proposition}
\begin{proof}
Let $\cat{N}_{\p(\cat{V})/\bf{H}_0}$ denote the normal 
bundle of the divisor ${\bf H}_0 \subset \p (\cat{V})$,
where $$\Psi :\cat{V} \longrightarrow \cat{M}_{\xi}: = \cat{M}(r, \alpha, \xi)$$ is the vector bundle (see \eqref{eq:a59}) in Proposition \ref{prop:3}.

Recall that the effective divisor $\bf{H}_0$ is numerically effective if and only if the restriction of the line bundle $\struct{\p (\cat{V})}(\bf{H}_0)$ to $\bf{H}_0$ is numerically effective.
From Poincar\'e adjunction formula we have the following 
\begin{equation}
\label{eq:a60}
\struct{\p (\cat{V})}(\bf{H}_0) \vert_{\bf{H}_0} \cong \cat{N}_{\p(\cat{V})/\bf{H}_0}.
\end{equation}

 Therefore, $\bf{H}_0$ is numerically effective if and only if the normal bundle $\cat{N}_{\p(\cat{V})/\bf{H}_0}$
 is numerically effective. Recall that  the tangent bundle $T \cat{M}_\xi$ is numerically effective if and only if the tautological line bundle $\struct{\p (T \cat{M}_\xi)}(1)$ is numerically 
effective.

Thus, to prove the proposition it is enough to show that 
 the normal bundle  $\cat{N}_{\p(\cat{V})/\bf{H}_0}$ is canonically isomorphic to the tautological 
line bundle $\struct{\p (T \cat{M}_\xi)}(1)$. 
 
First we show that the divisor $\bf{H}_0$ is canonically 
isomorphic to projective bundle $\p (T \cat{M}_\xi)$.
Let 
\begin{equation*}
\label{eq:a61}
\tilde{\Psi} : \p (\cat{V}) \longrightarrow \cat{M}_{\xi}
\end{equation*}
be the projective bundle. Let $E_* \in \cat{M}_\xi$, and 
\begin{equation}
\label{eq:a62}
\theta \in \tilde{\Psi}^{-1}(E_*) \cap \bf{H}_0 \subset 
\p (\cat{V}).
\end{equation}
Then, $\theta$ represents a hyperplane in the fibre 
$\cat{V}(E_*) = \pi_{\xi}^{-1}(E_*)^{\vee}$ of the vector bundle $\cat{V}$, where $\pi_\xi$ is defined in 
\eqref{eq:a21.1}. Let $H_{\theta}$ denote this hyperplane represented by $\theta$.
Note that $H_{\theta} \subset \pi_{\xi}^{-1}(E_*)^{\vee}$ and $\pi_{\xi}^{-1}(E_*)$ is the affine space modelled 
over the vector space $\coh{0}
{X}{\Omega^1_X(S) \otimes \SParend{E_*}},$ therefore
for any $f \in H_{\theta}$, we have 
$$\text{d}f \in (T^*_{E_*} \cat{M}_\xi)^* = T_{E_*}\cat{M}_\xi.$$
Since $\theta \in \bf{H}_0$, and $H_\theta$ is a hyperplane, the subspace of $T_{E_*} \cat{M}_{\xi}$
generated by $\{\text{d}f\}_{f \in H_{\theta}}$ is a
hyperplane in  $T_{E_*} \cat{M}_{\xi}$. Let $\tilde{\theta} \in \p (T_{E_*} \cat{M}_{\xi})$ denote this hyperplane.
Thus, we get a canonical isomorphism 
\begin{equation}
\label{eq:a63}
{\bf H}_0 \cong \p (T \cat{M}_{\xi})
\end{equation}
by sending $\theta$ to $\tilde{\theta}$.

Next, note that the fibre $\cat{N}_{\p(\cat{V})/\bf{H}_0}(\theta)$ of the normal bundle $\cat{N}_{\p(\cat{V})/\bf{H}_0}$ is canonically isomorphic to the quotient 
$\cat{V}(E_*) / H_{\theta}$. 
Consider the morphism 
$$\cat{V}(E_*) \longrightarrow T_{E_*}\cat{M}_{\xi}$$
of vector spaces defined by sending $f \mapsto \text{d} f$.
Since image of the hyperplane $H_{\theta}$ is contained in $\tilde{\theta}$, we have well defined morphism on quotients 
$$\cat{V}(E_*)/ H_{\theta} \longrightarrow T_{E_*} \cat{M}_{\xi}/ \tilde{\theta},$$
which is an isomorphism.
Recall that the fibre of the tautological line bundle 
$\struct{\p (T \cat{M}_{\xi})}(1)$ at $E_*$ is canonically 
identified with the quotient $T_{E_*} \cat{M}_{\xi}/ \tilde{\theta}$.
This completes the proof.
\end{proof}

Now consider the moduli space $\cat{M}'_{pc}(r, d, \alpha)$,
and the smooth divisor 
$${\bf H} = \p (\cat{W}) \setminus \cat{M}'_{pc}(r, d, \alpha)$$
as defined in \eqref{eq:a28}. Then using the same steps as in Proposition \ref{prop:4}, we can show the following.

\begin{proposition}
\label{prop:5}
The smooth divisor $\bf{H}$ is numerically effective 
if and only if the tangent bundle $T \cat{M}(r, d, \alpha)$ is numerically effective.
\end{proposition}

\section*{Acknowledgements} 
The author would like to thank the referee for 
comprehensive comments, pointing out some issues to be clarified in the first version of the paper and suggesting some appropriate references. 
The author is deeply grateful to Prof. Indranil Biswas
for useful discussions.

\end{document}